\documentclass{amsart}

\usepackage{amssymb}
\usepackage{amsmath}
\usepackage{enumerate}
\usepackage{array}
\usepackage{todonotes}
\usepackage{color}
\usepackage{dsfont} 
\usepackage{comment} 
\usepackage{pgfplots} 
\usetikzlibrary{patterns} 
\usepackage[T1]{fontenc}
\usepackage[utf8]{inputenc}
\usepackage{float}

\newtheorem{thm}{Theorem}
\newtheorem{lem}{Lemma}

\newtheorem{prop}{Proposition}

\theoremstyle{remark}

\newcommand{\CLT}{\textbf{CLT} }

\title[Some counterexamples to the central limit theorem...]{Some counterexamples to the central limit theorem for random rotations}
\author{Klaudiusz Czudek}
\address{Klaudiusz Czudek, Nicolaus Copernicus University,
Chopin Street 12/18,  87-100 Toru\'n, Poland}
\email{klaudiusz.czudek@gmail.com}

\subjclass[2020]{ Primary 37A50, 60F05.}

\keywords{random rotations, central limit theorem}

\begin{document}

\begin{abstract}
Fix an irrational number $\alpha$, and consider a random walk on the circle in which at each step one moves to $x+\alpha$ or $x-\alpha$ with probabilities $1/2, 1/2$ provided the current position is $x$. If an observable is given we can study a process called an additive functional of this random walk. One can formulate certain relations between the regularity of the observable and the Diophantine properties of $\alpha$ implying the central limit theorem. It is proven here that for every Liouville angle there exists a smooth observable such that the central limit theorem fails. We construct also a Liouville angle such that the central limit theorem fails with some analytic observable. For Diophantine angles some counterexample is given as well. An interesting question remained open. 
\end{abstract}

\maketitle

\section{Introduction}
Fix $\alpha\in \mathbb R$, and consider a Markov process $(Y_n^\alpha)_{n\ge 1}$ defined on some probability space $(\Omega, \mathcal F, \mathbb P)$ with the evolution governed by the transition kernel
\begin{equation}\label{E:1.1}
p(x, \cdot ) = \frac 1 2 \delta_{x+\alpha} + \frac 12 \delta_{x-\alpha}, \quad p : \mathbb S^1 \times \mathcal B (\mathbb S^1 ) \to [0,1],
\end{equation}
whose initial distribution, i.e. the distribution of $Y_1^\alpha$, is the Lebesgue measure (here  $\mathcal B (\mathbb S^1 )$ stands for the $\sigma$-algebra of Borel subsets of $\mathbb S^1$). One can easily verify the process is stationary. More work is needed to show the Lebesgue measure is the unique possible choice for the law of $Y_1^\alpha$ to make the process stationary (see e.g.  Theorem 7 and Remark 8 in \cite{Szarek_Zdunik_2016b}). In particular $(Y_n^\alpha)$ is ergodic, which means that if $A\in \mathcal B (\mathbb S^1 )$ is such that $p(x, A)=1$ for Lebesgue a.e. $x\in \mathbb S^1$ then $A$ is of the Lebesgue measure $0$ or $1$ (see e.g. Section 5 in \cite{Hairer_2006}, page 37, for characterizations of ergodicity and the relation to the notion of ergodicity in dynamical systems).



This paper is devoted to the central limit theorem (\textbf{CLT} for short) for additive functionals of $(Y_n^\alpha)$, i.e. processes of the form $\big(\varphi(Y^\alpha_1)+\cdots+\varphi(Y^\alpha_n)\big)$,
where a function $\varphi : \mathbb S^1 \to \mathbb R$ is usually called an observable. For convenience we assume that $\int \varphi(x)dx=0$. We say that \CLT holds for the process if
$$\frac{\varphi(Y^\alpha_1)+\cdots+\varphi(Y^\alpha_n)}{\sqrt n} \Rightarrow \mathcal N (0, \sigma) \quad \textrm{as $n\to \infty$}$$
for some $\sigma>0$. The validity of \CLT depends on Diophantine properties of $\alpha$. An angle $\alpha$ is called Diophantine of type $(c,\gamma)$, $c>0$, $\gamma\ge 2$ if
\begin{equation}\label{diophantine}
\bigg|\alpha - \frac p q \bigg| \ge \frac{c}{q^\gamma} \quad \textrm{for all $p, q\in \mathbb Z$, $q\not=0$.}
\end{equation}
An angle $\alpha$ is Liouville if it is not Diophantine of type $(c,\gamma)$ for any choice of $c>0$, $\gamma \ge 2$.

These and similar processes has been widely studied in the literature.
\begin{itemize}
\item Kesten \cite{Kesten_1960, Kesten_1961} investigated the limit distribution of
$$D_N(x,\alpha)=\sum_{n=0}^{N-1} \varphi(x+n\alpha) - N\int_{\mathbb{S}^1}\varphi(x)dx,$$
where $\varphi$ is the characteristic function of some interval and $(x,\alpha)$ is uniformly distributed in $\mathbb S^1\times \mathbb S^1$. This was later generalized to higher dimensions by Dolgopyat and Fayad \cite{Dolgopyat_Fayad_2014, Dolgopyat_Fayad_2020}.
\item Sinai and Ulcigrai \cite{Sinai_Ulcigrai_2008} considered a similar problem when $\varphi$ is non-integrable meromorphic function.
\item In the above examples a point in the space is chosen randomly thus one calls it a spatial \CLT. One can also fix a point in the space $x\in \mathbb S^1$, an angle $\alpha$ and, given $N$, pick randomly an integer number $n\in [1, N]$. The question arise what is the limit distribution of $D_n(x,\alpha)$ as $N$ is growing. This kind of limit theorems are called temporal. The first limit theorem in this flavour was proven by Beck \cite{Beck_2010, Beck_2011}. For further development see e.g. \cite{Dolgopyat_Sarig_2017}, \cite{Bromberg_Ulcigrai_2018}, \cite{Dolgopyat_Sarig_2020}.
\item Sinai \cite{Sinai_1999} considered a situation where one draws $+\alpha$ or $-\alpha$ with a probability distribution depending on the position in the circle (the method was to study a related random walk in random environment). He proved the unique ergodicity and stability of the process when $\alpha$ is Diophantine. Recently Dolgopyat et. al. \cite{Dolgopyat_Fayad_Saprykina_2021} studied the behaviour in the Liouvillean case.
\item Borda  \cite{Borda_2021} considered even a more general situation where several angles are given and one chooses one of them randomly. Given $p\in (0,1]$, he formulated certain Diophantine conditions implying \CLT for all $\varphi$ in the class of $p$-H{\"o}lder functions. Thus the author was concerned about what assumptions one should put on the angles of rotation to imply \CLT for all observables in a given class.
\end{itemize}

The situation here resembles the one from the last point, but here we rather touch the question how regular an observable should be to imply \CLT if $\alpha$ is given. Namely, using celebrated result by Kipnis and Varadhan \cite{Kipnis_Varadhan_1986} we prove the following statement.

\begin{prop}\label{P:1}
Let us assume $\alpha$ to be Diophantine of type $(c,\gamma)$, $\gamma \ge 2$. If a non-constant function $\varphi \in C^{r}$, $r>\gamma-1/2$ (possibly $r=\infty$), is such that $\int\varphi(x)dx=0$ then there exists $\sigma>0$ such that
$$\frac{\varphi(Y^\alpha_1)+\cdots+\varphi(Y^\alpha_n)}{\sqrt{n}} \Rightarrow \mathcal N (0, \sigma).$$
In particular, \CLT holds if $\alpha$ is Diophantine of an arbitrary type and $\varphi$ is smooth.
\end{prop}

\noindent The result is included for the sake of completeness, not because of novelty. This (or slightly different) statement has been proven independently by several people using various methods related to harmonic analysis (section 8 in \cite{Derriennic_Lin_2001}, section 7.5 in \cite{Weber_2009}, \cite{Zdunik_2017}, \cite{Borda_2021}).

By Proposition \ref{P:1} \textbf{CLT} holds if $\varphi$ is smooth and $\alpha$ is Diophantine of an arbitrary type. It is natural to ask then if for every Liouville $\alpha$ there exists a smooth $\varphi$ for which \CLT fails. It is also natural to ask if \CLT fails if analytic observables are considered. This leads us to the following theorems showing dichotomy between the behaviour of Liouville and Diophantine random rotation, similar to the one appearing in smooth conjugacy results for circle diffeomorphisms (see the beginning of Chapter I.3 in \cite{deMelo_vanStrien_1993}).


\begin{thm}\label{T:1}
There exists an irrational $\alpha$ and $\varphi \in C^\omega(\mathbb S^1)$ such that \CLT fails.
\end{thm}

\noindent Note that by Proposition \ref{P:1} the angle in the assertion must be Liouville.

\begin{thm}\label{T:2}
Let $\alpha$ be an irrational number. Let us assume there exist $c>0$, $\gamma>5$ such that
$$\bigg|\alpha - \frac{p}{q} \bigg| \le \frac{c}{q^\gamma} \quad \textrm{for infinitely many $p,q \in \mathbb Z$, $q\not = 0$.} $$
Let $r$ be the largest positive integer with $r<\frac{\gamma}{2}-\frac 3 2$. Then there exist $\varphi \in C^r$ such that \CLT fails.
\end{thm}

\noindent The only reason for making the assumption $\gamma>5$ is to ensure $\frac{\gamma}{2}-\frac 3 2$ greater than $1$, so that the condition $r<\frac{\gamma}{2}-\frac 3 2$ is satisfied for at least one positive integer $r$. A slightly changed proof of Theorem \ref{T:2} yields the following.

\begin{thm}\label{T:3}
Let $\alpha$ be Liouville. Then there exists $\varphi\in C^\infty(\mathbb S^1)$ such that \CLT fails.
\end{thm}

Let us end this section with an interesting open problem. An angle $\alpha$ is called badly approximable when it is Diophantine of type $(c, 2)$ for some $c>0$ (for instance, every quadratic irrational is badly approximable). Proposition \ref{P:1} yields if $\varphi$ is $C^2$ then the additive functional satisfies \textbf{CLT}. Unfortunately, Theorem \ref{T:2} does not give any counterexample in that case. This leads to a natural question: does \CLT holds if $\alpha$ is badly approximable (e.g. $\alpha$ is the golden ratio) and $\varphi$ is $C^1$?

\section{The Poisson equation and central limit theorem}\label{S:3}

One of methods of proving \CLT for additive functionals of Markov chains is the Gordin-Lif\v{s}ic method \cite{Gordin_Lifsic_1978}, which is to be roughly explained in present section (note that in \cite{Weber_2009}, \cite{Zdunik_2017}, \cite{Borda_2021} different techniques have been used). Before that let us define the operator
$$
T\varphi(x)=\frac 1 2 \varphi(x+\alpha) + \frac 12 \varphi(x-\alpha), \quad \varphi\in B(\mathbb S^1), \ T: B(\mathbb S^1)\rightarrow B(\mathbb S^1),
$$
where $B(\mathbb S^1)$ is the space of Borel measurable functions. By the very definition of a Markov process, if $(Y_n^\alpha)$ is defined on $(\Omega, \mathcal F, \mathbb P)$ then
\begin{equation}\label{E:dual}
\mathbb E ( \varphi(Y^\alpha_{n+1} ) | Y_n^\alpha ) = \int_{\mathbb{S}^1} p( Y_n^\alpha, dy) \varphi(y) = T\varphi (Y_n^\alpha), \quad n\ge 1,
\end{equation}
where $p$ is the transition function (\ref{E:1.1}).

Let $\varphi : \mathbb S^1 \rightarrow \mathbb R$ be a square integrable function (with respect to the Lebesgue measure) with $\int \varphi(x)dx=0$. To show the convergence of $\frac{1}{\sqrt{n}}(\varphi(Y_1)+\cdots+\varphi(Y_n))$ to the normal distribution we solve so called Poisson equation\footnote{In dynamical systems theory this equation (with $T$ replaced by a Koopman operator) is called a cohomological equation. The name ``Poisson equation'' is more common in theory of stochastic processes, probably due to the fact that writing down the corresponding equation for a Brownian motion, which is a continuous time Markov process, gives $\frac 1 2 \Delta \varphi = \psi$, where $\Delta$ is the Laplace operator. Note $\frac 1 2 \Delta$ is the infinitesimal generator of the Brownian motion.} $T\psi - \psi =\varphi$, where $\psi\in L^2(\mathbb S^1)$ is unknown. If the solution $\psi$ exists then we can write
$$\varphi(Y_1)+\cdots+\varphi(Y_n)$$
\begin{equation}\label{E:P1.1}
=\big[(T\psi(Y_1)-\psi(Y_2))+\cdots+(T\psi(Y_{n-1})-\psi(Y_n))\big]+(T\psi(Y_n) - \psi (Y_1)).
\end{equation}
When divided by $\sqrt{n}$, the second term tends to zero in probability. It is sufficient then to show \CLT for the first process, which is an ergodic, stationary martingale by (\ref{E:dual}). For such processess \CLT is valid (see \cite{Brown_1971}). Thus the assertion follows provided the solution of the Poisson equation exists.

Observe that $(I-T) u_n=(1-\cos(2\pi n \alpha)) u_n$ for $u_n(x)=\exp(2\pi i n x)$, $x\in \mathbb S^1$, $n\in\mathbb Z$. Therefore the trigonometric system $(u_n)_{n\in\mathbb Z}$ is also the orthonormal system of eigenvectors of $I-T$ with corresponding eigenvalues $1-\cos(2\pi n \alpha)$, $n\in\mathbb Z$. We deduce the $n$-th Fourier coefficient of $(I-T)\psi$, $\psi \in L^2(\mathbb S^1)$, is of the form $(1-\cos(2\pi n \alpha))\hat{\psi}(n)$, $n\in \mathbb Z$. This yields a recipe to find $\psi$ when $\varphi$ is given. Namely, $\psi$ should be a square integrable function whose Fourier series coefficient are

\begin{equation}\label{fourier}
\hat{\psi}(n) = \frac{|\hat{\varphi}(n)|}{1-\cos(2\pi\alpha n)}, \quad n\in\mathbb Z \setminus \{0\},
\end{equation}
while $\hat{\psi}(0)$ is an arbitrary real number. Note we use here also the assumption that $\hat{\varphi}(0)=\int\varphi(x)dx=0$. Indeed, $1-\cos(0)=0$ implies that we must have $\hat{\varphi}(0)=0$ to solve the equation. What remains to do is to show the convergence
\begin{equation}\label{condition1}
\sum_{n\in\mathbb Z\setminus \{0\}} \frac{|\hat{\varphi}(n)|^2}{(1-\cos(2\pi\alpha n))^2}<\infty,
\end{equation}
to make sure the object with Fourier coefficients (\ref{fourier}) is indeed a square integrable function.

In fact the solution of the Poisson equation does not have to exists to have \textbf{CLT}. Note the processes under consideration are reversible, which means that the distribution of random vectors $(Y^\alpha_1, \ldots, Y^\alpha_n)$ and $(Y^\alpha_n, \ldots, Y^\alpha_1)$ are the same for every natural $n$ or, equivalently, the operator $T$ is self-adjoint. In celebrated paper \cite{Kipnis_Varadhan_1986} (see Theorem 1.3 therein) the authors have proven the condition $\varphi\in \textrm{Im}(I-T)$ can be relaxed to
\begin{equation}\label{Kipnis_Varadhan}
\varphi \in \textrm{Im}(\sqrt{I-T}),
\end{equation}
where $\sqrt{I-T}$ is the square root of $I-T$ (recall the square root of a positive semidefinite, self-adjoint operator $P$ acting on a Hilbert space is the operator $\sqrt{P}$ with the property $(\sqrt{P})^2=P$). Since the $n$-th Fourier coefficient of the function $(I-T)\psi$, $\psi \in L^2(\mathbb S^1)$ is given by $(1-\cos(2\pi n \alpha))\hat{\psi}(n)$, we easily deduce that $\sqrt{I-T}$ is well defined on $L^2(\mathbb S^1)$ and the $n$-th Fourier coefficient of the function $\sqrt{I-T}\psi$, $\psi \in L^2(\mathbb S^1)$, is given by $\sqrt{1-\cos(2\pi n \alpha)}\hat{\psi}(n)$. Thus (\ref{Kipnis_Varadhan}) leads to the condition
\begin{equation}\label{condition2}
\sum_{n\in\mathbb Z \setminus \{0\}} \frac{|\hat{\varphi}(n)|^2}{1-\cos(2\pi\alpha n)}<\infty,
\end{equation}
weaker than (\ref{condition1}). Moreover, \cite{Kipnis_Varadhan_1986} (see (1.1) therein) delivers a formula for $\sigma$, which reads here as
$$\sigma^2=\sum_{n\in \mathbb Z \setminus \{0\}} \frac{1+\cos(2\pi \alpha n)}{1-\cos(2\pi \alpha n)} |\hat{\varphi}(n)|^2.$$
Clearly, $\sigma^2<\infty$ if (\ref{condition2}) is satisfied and $\sigma>0$ if $\varphi$ is non-constant.

We are in position to prove Proposition \ref{P:1}. We recall the statement for the convenience of the reader.

\setcounter{prop}{0}

\begin{prop}
Let us assume $\alpha$ to be Diophantine of type $(c,\gamma)$, $\gamma \ge 2$. If a non-constant function $\varphi \in C^{r}$, $r>\gamma-1/2$ (possibly $r=\infty$), is such that $\int\varphi(x)dx=0$ then there exists $\sigma>0$ such that
$$\frac{\varphi(Y^\alpha_1)+\cdots+\varphi(Y^\alpha_n)}{\sqrt{n}} \Rightarrow \mathcal N (0, \sigma).$$
In particular, \CLT holds if $\alpha$ is Diophantine of an arbitrary type and $\varphi$ is smooth.
\end{prop}
\begin{proof}
We are going to prove (\ref{condition2}) is satisfied. Fix $\alpha$ and $\varphi$ as above. Clearly $\sum_{n\in \mathbb Z} |\hat{\varphi}(n)|^2<\infty$ since $\varphi$ is square integrable, therefore the problem is when $\cos(2\pi \alpha n)$ is close to 1, which happens exactly when $\alpha n$ is close to some integer. To handle this we will use the fact that $\alpha$ is Diophantine of type $(c, \gamma)$. This means
\begin{equation}\label{E:5.1}
\bigg|\alpha - \frac{p}{n}\bigg| \ge \frac{c}{n^\gamma} \quad \textrm{for all $p, n\in \mathbb Z$, $n\not = 0$.}
\end{equation}
By Taylor's formula $|\cos(2 \pi (p+x))-1|=\frac{(2\pi x)^2}{2}+o(x^2)$ for $p\in \mathbb Z$. As a consequence there exists $\eta>0$ such that 
$$\big|\cos(2 \pi \alpha n)-1 \big|\ge  2\pi \eta|n\alpha - p|^2 \ge \frac{2\pi \eta c^2}{n^{2(\gamma-1)}}$$
for an arbitrary $n \in\mathbb Z$. If $\varphi \in C^r$ then $\hat{\varphi}(n)\le C|n|^{-r}$ for some constant $C$ thus
$$ \frac{|\hat{\varphi}(n)|^2}{1-\cos(2\pi\alpha n)} \le \frac{C^2}{2\pi \eta c^2} |n|^{-2r+2(\gamma-1)}$$
for every $n$. It is immediate that if $r>\gamma-\frac 12$, then the series (\ref{condition2}) is convergent. This implies \CLT by Theorem 1.3 in \cite{Kipnis_Varadhan_1986}.
\end{proof}

Clearly, if $\varphi$ is a trigonometric polynomial, then series (\ref{condition2}) becomes a finite sum and thus the condition is trivially satisfied. This yields another proposition, which will be used in the proof of Theorem \ref{T:2}.

\begin{prop}
\label{P:2}
Let us assume $\alpha$ to be irrational. If $\varphi$ is a non-constant trigonometric polynomial with $\int\varphi(x)dx=0$ then there exists $\sigma>0$ such that
$$\frac{\varphi(Y^\alpha_1)+\cdots+\varphi(Y^\alpha_n)}{\sqrt{n}} \Rightarrow \mathcal N (0, \sigma).$$
\end{prop}

\section{Auxiliary results}
In the proofs three lemmas will be pivotal. Given integer $q\ge 1$, $\eta\in (0,1/2)$, define $G_q^\eta$ to be the subset of $\mathbb S^1$ containing all points whose distance from the set $\{ 0, \frac{1}{q}, \ldots, \frac{ q-1}{q} \}$ (where $\cos(2\pi q x)$ attains value 1) is less than $\frac {\eta}{q}$. Clearly $\textrm{Leb}(G_q^\eta)=2\eta$ whatever $q$ is. Recall that $(Y^{\alpha}_n)$ stands for the Markov process defined on some probability space $(\Omega, \mathcal F, \mathbb P)$ with transition function (\ref{E:1.1}) and $Y_1^\alpha \sim \textrm{Leb}$.

\begin{lem}\label{L:1}
Let $\alpha=\frac p q$, $\varphi(x)=2^{-q} \cos(2 \pi q x)$ and let $s\in (0,1)$. Let $N$ be an arbitrary natural number with $2^{-q-1}N^{1-s}>2$. If $\alpha'$ is sufficiently close to $\alpha$ then
$$\mathbb P \bigg(\frac{\varphi(Y_1^{\alpha'})+\cdots+\varphi(Y_N^{\alpha'})}{N^{s}} > 2 \bigg) > \frac{1}{6}.$$
\end{lem}
\noindent Note the assertion is more difficult to obtain when $s$ is close to 1.
\begin{proof}
The result is the consequence of the invariance of $\varphi$ under the action of the rotation of angle $\alpha$. In particular the set $G_q^\eta$ is invariant for every $\eta>0$. Take $N$ like in the statement, and choose $\alpha'$ so close to $\alpha$ that $x+n\alpha' \in G_q^{1/6}$ for $|n|\le N$ and $x\in G_q^{1/12}$. By the definition of $G_q^\eta$, the value of $\varphi$ on $G_q^{1/6}$ is greater or equal to $2^{-q}\cos(2\pi/6)\ge 2^{-q}\cdot 1/2$. Thus $\varphi(x+n\alpha')\ge 2^{-q}\cdot 1/2=2^{-q-1}$ for $|n|\le N$ and $x\in G_q^{1/12}$. This yields

$$\{ Y^{\alpha'}_1 \in G_q^{1/12} \} \subseteq \bigg\{ \frac{\varphi(Y_1^{\alpha'})+\cdots+\varphi(Y_N^{\alpha'})}{N} > 2^{-q-1} \bigg\}$$
$$ =  \bigg\{ \frac{\varphi(Y_1^{\alpha'})+\cdots+\varphi(Y_N^{\alpha'})}{N^{s}} > 2^{-q-1}N^{1-s} \bigg\}$$

\noindent Using the facts that $Y_1^{\alpha'}\sim \textrm{Leb}$, $\textrm{Leb}(G_q^{1/12})=1/6$ and $2^{-q-1}N^{1-s}>2$ we have
$$\mathbb P \bigg(\frac{\varphi(Y_1^{\alpha'})+\cdots+\varphi(Y_N^{\alpha'})}{N^{s}} > 2 \bigg)$$
$$\ge \mathbb P \bigg(  \frac{\varphi(Y_1^{\alpha'})+\cdots+\varphi(Y_N^{\alpha'})}{N^{s}} > 2^{-q-1}N^{1-s}  \bigg)
\ge \mathbb P (Y_1^{\alpha'} \in G_q^{1/12}) = \frac{1}{6},$$
which yields the assertion.
\end{proof}

A slightly different lemma is the following.
\begin{lem}\label{L:3}
Let $\alpha$ be an irrational number, $s\in (1/2, 1)$, $c>0$, $\gamma\ge 2$. If $\alpha$ satisfies
$$\bigg|\alpha - \frac p q \bigg|\le \frac{c}{q^\gamma}$$
for some pair of integers $p,q$, $q\not= 0$, then
$$\mathbb P \bigg( \frac{\varphi(Y_1^{\alpha})+\cdots+\varphi(Y_N^{\alpha})}{N^s} > \frac{\sqrt{2}}{2\cdot (16 c)^{1-s}} \bigg) > \frac{1}{8},$$
where $\varphi(x)=q^{-(\gamma-1)(1-s)}\cos (2\pi q x)$, $N=\lfloor\frac{q^{\gamma-1}}{16c}\rfloor$.
\end{lem}
\begin{proof}
If $|\alpha - \frac p q|\le \frac{c}{q^\gamma}$ and $|k|\le \frac{q^{\gamma-1}}{16c}$ then
\begin{equation}\label{L:3.1}
\bigg|k\alpha- k\frac p q \bigg|\le |k| \frac{c}{q^\gamma} < \frac{1}{16q}.
\end{equation}
Thus $z+n\alpha \in G^{1/8}_q$ for all $z\in G^{1/16}_q$ and integers $n$ with $|n|\le N$. On the other hand, the value of $\varphi$ on $G^{1/8}_q$ is greater or equal to  $q^{-(\gamma-1)(1-s)}\cos(\frac{2 \pi}{8})=\frac{\sqrt{2}}{2}\cdot q^{-(\gamma-1)(1-s)}$. By the same reasoning as in the proof of Lemma \ref{L:1} we have
$$\{Y_1^\alpha \in G^{1/16}_q \} \subseteq \bigg\{  \frac{\varphi(Y_1^{\alpha})+\cdots+\varphi(Y_N^{\alpha})}{N^s} > \frac{\sqrt{2}}{2\cdot (16c)^{1-s}} \bigg\}$$
and consequently
$$\mathbb P \bigg( \frac{\varphi(Y_1^{\alpha})+\cdots+\varphi(Y_N^{\alpha})}{N^s} > \frac{\sqrt{2}}{2\cdot (16c)^{1-s}} \bigg)\ge \mathbb P (Y_1^\alpha\in G^{1/16}_q) =\frac{1}{8}.$$

\end{proof}

Take $\alpha=p/q$ rational ($p/q$ is in the irreducible form) and the corresponding process $(Y_n^\alpha)$. If the initial point $Y_1^\alpha$ is already known, then we know also each $Y^\alpha_n$, $n\in \mathbb N$, is contained almost surely in the orbit of $Y_1^\alpha$ under the action of the rotation of angle $\alpha$, $\{Y_1^\alpha, Y_1^\alpha+\alpha, \ldots, Y_1^\alpha+(q-1)\alpha\}$ (this set is finite, since $\alpha$ is rational). The process $(Y^\alpha_n)$ can be therefore treated as a finite state Markov chain.

 If $q$ is odd, then the process $(Y^\alpha_n)$ treated as a finite state Markov chain is aperiodic and irreducible. Its stationary distribution the uniform distribution on the set $\{Y_1^\alpha, Y_1^\alpha+\alpha, \ldots, Y_1^\alpha+(q-1)\alpha\}$ (every state is of measure $1/q$). It follows from Theorem 8.9 (page 131) \cite{Billingsley_1995} that
\begin{equation}\label{E:exp.conv.}
|\mathbb P (Y^\alpha_n = Y_1^\alpha+i\alpha) - 1/q| \le A \rho^n \quad \textrm{for $i=0,1, \ldots, q-1$},
\end{equation}
where the constants $A$ and $\rho\in(0,1)$ are independent of $x$ (since neither the space nor the transition probabilities depend on $x$). Let $\varphi(x)=a\cos(2\pi q' x)$ for some $a>0$ and $q'$ not a multiplicity of $q$. Since $p/q$ is assumed to be in an irreducible form, $p/q\cdot q'$ is not an integer and thus we have
$$1/q\sum_{i=0}^{q-1} \varphi(x+i\alpha)=0$$
for every $x\in \mathbb S^1$, which is equivalent to say that the integral of $\varphi$ with respect to the stationary distribution of $(Y_n^\alpha)$ (treated as a finite state Markov chain) equals zero. Moreover, using (\ref{E:exp.conv.}) gives
$$\bigg| \mathbb E\big( \varphi(Y^\alpha_n) \ \big| \ Y^\alpha_1 \big) \bigg|
= \bigg| \sum_{i=0}^{q-1} \mathbb P ( Y^\alpha_n = Y^\alpha_1+i\alpha  ) \cdot \varphi(Y^\alpha_1+i\alpha) - 1/q\sum_{i=0}^{q-1} \varphi(Y^\alpha_1+i\alpha ) \bigg|$$
$$\le \sum_{i=0}^{q-1} \|\varphi\|_\infty \big|  \mathbb P ( Y^\alpha_n = Y^\alpha_1+i\alpha  ) - 1/q \big|
\le A q \|\varphi \|_\infty \rho^n$$
for $n\ge 1$. Thus
\begin{equation}\label{E:3.1}
\sum_{n=1}^\infty \bigg| \mathbb E\big( \varphi(Y^\alpha_n) \ \big| \ Y^\alpha_1 \big) \bigg| \le \frac{A q \|\varphi \|_\infty}{1-\rho} \quad \textrm{a.s.}
\end{equation}
The next lemma is essentially the consequence of the central limit theorem for finite state irreducible and aperiodic Markov chains. However, using (\ref{E:3.1}) we may deduce it in simpler way.

\begin{lem}\label{L:2}
Let $\alpha=\frac p q$ be rational (in irreducible form), $q$-odd. Let  $\varphi(x)=a\cos(2\pi q' x)$ for some $a>0$ and $q'$ not a multiplicity of $q$. If $s>1/2$ then for every $\varepsilon>0$ and $\delta>0$ there exists $N$ such that
$$\mathbb P \bigg( \frac{\big| \varphi(Y^\alpha_1)+\cdots+\varphi(Y^\alpha_n)\big|}{n^s} > \delta \bigg) < \varepsilon$$
for $n \ge N$.
\end{lem}
\begin{proof}
It follows from the Chebyshev inequality. We have
$$ \mathbb P \bigg( \frac{\big|\varphi(Y^\alpha_1)+\cdots+\varphi(Y^\alpha_n)\big|}{n^s} > \delta \bigg) 
\le \frac{\mathbb E \big( \varphi(Y^\alpha_1)+\cdots+\varphi(Y^\alpha_n) \big)^2}{\delta^2 n^{2s}} $$
$$= \frac{ \big|  \mathbb E \big(\sum_{i=1}^n \varphi^2(Y_i^\alpha) + 2\sum_{i=1}^{n-1} \varphi(Y^\alpha_i)(\varphi(Y^\alpha_{i+1})+\cdots+\varphi(Y^\alpha_n) ) \big)\big|}{\delta^2 n^{2s}}$$
$$\le \sum_{i=1}^n \frac{\mathbb E\varphi^2(Y_i^\alpha)}{\delta^2 n^{2s}} + 2\sum_{i=1}^n \frac{\big| \mathbb E \big(\varphi(Y^\alpha_i) \mathbb E[ \varphi(Y^\alpha_{i+1})+\cdots+\varphi(Y^\alpha_n) | Y^\alpha_i ]\big)\big|}{\delta^2n^{2s}}$$
$$=\frac{1}{\delta^2 n^{2s-1}} \int\varphi^2(x)dx + 2\sum_{i=1}^n \frac{\big| \mathbb E \varphi(Y^\alpha_i) \big|\cdot \big| \mathbb E[ \varphi(Y^\alpha_{i+1})+\cdots+\varphi(Y^\alpha_n) | Y^\alpha_i ]\big|}{\delta^2n^{2s}}.$$
$$\le\frac{1}{\delta^2 n^{2s-1}} \int\varphi^2(x)dx + 2 \sum_{i=1}^n \frac{\big| \mathbb E \varphi(Y^\alpha_i) \big|\cdot \bigg( \big|\mathbb E[ \varphi(Y^\alpha_{i+1})| Y^\alpha_i ]\big|  +\cdots+ \big|\mathbb E[ \varphi(Y^\alpha_n) | Y^\alpha_i ]\big|\bigg)}{\delta^2n^{2s}}.$$
By (\ref{E:3.1})  and the stationarity of the process each of the numerators in the sum does not exceed $\frac{2A q \|\varphi \|^2_\infty}{1-\rho}$, thus the second term is bounded by
$$n\cdot \frac{2A q \|\varphi \|^2_\infty}{(1-\rho)\delta^2n^{2s}}=\frac{2A q \|\varphi \|^2_\infty}{(1-\rho)\delta^2n^{2s-1}}.$$
The entire expression tends to zero since $s>1/2$. The assertion follows.
\end{proof}

\section{Proof of Theorem \ref{T:1}}

Fix an arbitrary $s\in (\frac 1 2, 1)$. We are going to construct an angle $\alpha$ and an observable $\varphi$ with $\int\varphi(x)dx=0$ such that there exist infinitely many $n$'s with
$$\mathbb P \bigg( \frac{\varphi(Y_1^\alpha)+\cdots+\varphi(Y_n^\alpha)}{n^{s}} > 1 \bigg) > \frac{1}{12}.$$
Consequently the process does not satisfy \CLT since \CLT would imply the above quantity tends to zero. First we shall define inductively a sequence of numbers $\alpha_k$ convergent to some $\alpha$ along with certain observables $\varphi_k$. Then we will put $\varphi=\sum_k \varphi_k$ and use some relations between $\alpha_k$ and $\varphi_k$ established during the induction process to get the above assertion.

Put $\alpha_1=\frac 1 3=\frac {p_1}{q_1}$ (when we represent a rational number as a fraction of integers we always assume it to be in an irreducible form, so here $p_1=1$ and $q_1=3$), and set $\varphi_1(x)=2^{-q_1} \cos(2\pi q_1 x)$. Take $N_1$ so large that $2^{-q_1-1}N_1^{1-s}>2$ and apply Lemma \ref{L:1} to obtain an angle $\alpha_2=\frac{p_2}{q_2}$, with $q_2>q_1$ and $q_2$ odd, such that
\begin{equation}\label{E:4.1}
\mathbb P \bigg( \frac{\varphi_1(Y_1^{\alpha_2})+\cdots+\varphi_1(Y_{N_1}^{\alpha_2})}{N_1^{s}} >2 \bigg) >\frac 1 6.
\end{equation}
Define $\varphi_2(x)=2^{-q_2}\cos(2\pi q_2 x)$. Take $N_2>N_1$ so large that $2^{-q_2-1} N_2^{1-s}>2$. Clearly $q_1$ is not a multiplicity of $q_2$, hence by Lemma \ref{L:2} we can assume that $N_2$ is so large that
\begin{equation}\label{E:4.11}
\mathbb  P \bigg( \frac{\big|\varphi_1(Y_1^{\alpha_2})+\cdots+\varphi_1(Y_{N_2}^{\alpha_2})\big|}{N_2^{s}} >\frac 1 4 \bigg) < \frac 1 4 \cdot \frac{1}{6}.\end{equation}
Again use Lemma \ref{L:1} to obtain an angle $\alpha_3=\frac{p_3}{q_3}$, with $q_3>q_2$ and $q_3$ odd, such that
\begin{equation}\label{E:4.13}
\mathbb  P \bigg( \frac{\varphi_2(Y_1^{\alpha_3})+\cdots+\varphi_2(Y_{N_2}^{\alpha_3})}{N_2^{s}} > 2 \bigg) > \frac{1}{6}.
\end{equation}
We assume also the number $\alpha_3$ is so close to $\alpha_2$ that (\ref{E:4.1}) and (\ref{E:4.11}) still hold with $\alpha_2$ replaced by $\alpha_3$. This combined with (\ref{E:4.13}) gives
$$\mathbb P \bigg( \frac{\varphi_i(Y_1^{\alpha_3})+\cdots+\varphi_i(Y_{N_i}^{\alpha_3})}{N_i^{s}} >2 \bigg) > \frac 1 6, \quad \textrm{for $i=1,2$,}$$
and
$$\mathbb  P \bigg( \frac{\big|\varphi_1(Y_1^{\alpha_3})+\cdots+\varphi_1(Y_{N_2}^{\alpha_3})\big|}{N_2^{s}} >\frac 1 4 \bigg) < \frac 1 4 \cdot \frac{1}{6}.$$

Assume $\alpha_k=\frac{p_k}{q_k}$, $N_i$, $\varphi_i$ are already defined, $k\ge 3$, $i<k$. These objects satisfy the relations
\begin{equation}\label{E:4.51}
\mathbb  P \bigg( \frac{\big|\varphi_i(Y_1^{\alpha_{k}})+\cdots+\varphi_i(Y_{N_j}^{\alpha_{k}})\big|}{N_j^{s}} >\frac{1}{4^i} \bigg) < \frac{ 1}{4^i} \cdot \frac{1}{6} \quad \textrm{for $j=1,\ldots, k-1$, $i<j$,}
\end{equation}
and
\begin{equation}\label{E:4.41}
\mathbb  P \bigg( \frac{\varphi_i(Y_1^{\alpha_{k}})+\cdots+\varphi_i(Y_{N_i}^{\alpha_{k}})}{N_i^{s}} > 2 \bigg) > \frac{1}{6} \quad \textrm{for $i=1,\ldots, k-1$.}
\end{equation}

\noindent Define $\varphi_k(x)=2^{-q_k} \cos(2 \pi q_k x)$ and take $N_k>N_{k-1}$ so large that $2^{-q_k-1} N_k^{1-s}>2$ and
\begin{equation}\label{E:4.2}
\mathbb  P \bigg( \frac{\big|\varphi_i(Y_1^{\alpha_k})+\cdots+\varphi_i(Y_{N_k}^{\alpha_k})\big|}{N_k^{s}} >\frac{1}{4^i} \bigg) < \frac{ 1}{4^i} \cdot \frac {1}{6}
\end{equation}
for $i=1,\ldots, k-1$, by Lemma \ref{L:2}. Use Lemma \ref{L:1} to get a number $\alpha_{k+1}=\frac{p_{k+1}}{q_{k+1}}$, with $q_{k+1}>q_k$, $q_{k+1}$ odd, such that
\begin{equation}\label{E:4.3}
\mathbb  P \bigg( \frac{\varphi_k(Y_1^{\alpha_{k+1}})+\cdots+\varphi_k(Y_{N_k}^{\alpha_{k+1}})}{N_k^{s}} > 2 \bigg) > \frac{1}{6}.
\end{equation}
We should take care that $\alpha_{k+1}$ is so close to $\alpha_k$ that  (\ref{E:4.51}), (\ref{E:4.41}) and (\ref{E:4.2}) still hold with $\alpha_k$ replaced by $\alpha_{k+1}$. With this modification, (\ref{E:4.51}) and (\ref{E:4.2}) become
\begin{equation}\label{E:4.5}
\mathbb  P \bigg( \frac{\big|\varphi_i(Y_1^{\alpha_{k+1}})+\cdots+\varphi_i(Y_{N_j}^{\alpha_{k+1}})\big|}{N_j^{s}} >\frac{1}{4^i} \bigg) < \frac{ 1}{4^i} \cdot \frac{1}{6} \quad \textrm{for $j=1,\ldots, k$, $i<j$}.
\end{equation}
while (\ref{E:4.41}) and (\ref{E:4.3}) can be rewritten as
\begin{equation}\label{E:4.4}
\mathbb  P \bigg( \frac{\varphi_i(Y_1^{\alpha_{k+1}})+\cdots+\varphi_i(Y_{N_i}^{\alpha_{k+1}})}{N_i^{s}} > 2 \bigg) > \frac{1}{6} \quad \textrm{for $i=1,\ldots, k$.}
\end{equation}
This completes the induction. Observe there is no inconsistency in assuming that $q_{k+1}$'s grow so fast that
\begin{equation}\label{E:4.8}
2^{-q_{k+1}} N_i^{1-s}<4^{-(k-i)} \quad \textrm{for $i=1,\ldots k$.}
\end{equation}

This way the sequences of numbers $(\alpha_k)$, $(N_k)$ and functions $(\varphi_k)$ are defined. Set $\alpha=\lim_{k\to \infty} \alpha_k$ and $\varphi=\sum_{k=1}^\infty \varphi_k$.
When passing to the limit, inequality (\ref{E:4.5}) becomes
\begin{equation}\label{E:4.6}
\mathbb  P \bigg( \frac{\big|\varphi_i(Y_1^{\alpha})+\cdots+\varphi_i(Y_{N_j}^{\alpha})\big|}{N_j^{s}} \ge \frac{1}{4^i} \bigg) \le \frac{ 1}{4^i} \cdot \frac {1}{6} \quad \textrm{for $j>1$, $i<j$}.
\end{equation}
while (\ref{E:4.4}) yields
\begin{equation}\label{E:4.7}
\mathbb  P \bigg( \frac{\varphi_i(Y_1^{\alpha})+\cdots+\varphi_i(Y_{N_i}^{\alpha})}{N_i^{s}} \ge 2 \bigg) \ge \frac{1}{6} \quad \textrm{for $i\ge 1$.}
\end{equation}

The function $\varphi$ is analytic. Indeed, by design
$$\varphi(x) = \sum_{k=-\infty}^\infty c_k e^{2\pi i k x},$$
where $c_k= \|\varphi_j\|_\infty=2^{-q_j}$ if $|k|=q_j$ and zero otherwise. Thus the Fourier coefficients of $\varphi$ decay exponentially fast, which implies $\varphi$ to be analytic\footnote{Indeed, $\varphi$ is defined as a series on the circle, however by the exponential convergence it can be extended to some neighbourhood of the unit disc $\mathbb D$ in the complex plane $\mathbb{C}$. Then $\varphi$ becomes a sum of holomorphic functions convergent uniformly on compact subsets of the domain of $\varphi$. Theorem 10.28 (page 214) in \cite{Rudin_1987} implies $\varphi$ is holomorphic.}. Obviously $\int \varphi(x)dx=0$ by the Lebesgue convergence theorem. Observe also that (\ref{E:4.8}) combined with $\|\varphi_i\|_\infty=2^{-q_i}$ yield
\begin{equation}\label{E:4.9}
\sum_{i>k} \|\varphi_i\|_\infty N_k^{1-s}<\sum_{i=1}^\infty 4^{-i}=\frac 1 2.
\end{equation}

We are in position to complete the proof, i.e. to show that
$$\mathbb P \bigg( \frac{\varphi(Y_1^{\alpha})+\cdots+\varphi(Y_{N_k}^{\alpha})}{N_k^{s}} \ge 1 \bigg) \ge \frac {1}{12}$$
for every $k$. To this end fix $k$ and write
$$\frac{\varphi(Y_1^{\alpha})+\cdots+\varphi(Y_{N_k}^{\alpha})}{N_k^{s}} = \sum_{i\le k} \frac{\varphi_i(Y_1^{\alpha})+\cdots+\varphi_i(Y_{N_k}^{\alpha})}{N_k^{s}}$$
$$+ \sum_{i>k} \frac{\varphi_i(Y_1^{\alpha})+\cdots+\varphi_i(Y_{N_k}^{\alpha})}{N_k^{s}}.$$
From (\ref{E:4.9}) it easily follows that the absolute value of the second summand on the right-hand side is less than $\frac 1 2$ almost surely. Thus
$$\mathbb P \bigg ( \frac{\varphi(Y_1^{\alpha})+\cdots+\varphi(Y_{N_k}^{\alpha})}{N_k^{s}} \ge 1 \bigg )
\ge \mathbb P \bigg( \sum_{i\le k} \frac{\varphi_i(Y_1^{\alpha})+\cdots+\varphi_i(Y_{N_k}^{\alpha})}{N_k^{s}} \ge 3/2 \bigg)$$
$$ \ge \mathbb P \bigg ( \frac{\varphi_k(Y_1^{\alpha})+\cdots+\varphi_k(Y_{N_k}^{\alpha})}{N_k^{s}}\ge 2 \bigg )
-\sum_{i<k} \mathbb P \bigg ( \frac{\big|\varphi_i(Y_1^{\alpha})+\cdots+\varphi_i(Y_{N_k}^{\alpha})\big|}{N_k^{s}} \ge \frac{1}{4^i} \bigg). $$
By  (\ref{E:4.6}) and (\ref{E:4.7}) it follows that
$$\mathbb P \bigg ( \frac{\varphi(Y_1^{\alpha})+\cdots+\varphi(Y_{N_k}^{\alpha})}{N_k^{s}} \ge 1 \bigg ) \ge \frac{1}{6} - \sum_{i=1}^\infty\frac{1}{4^i} \cdot \frac{1}{6} = \frac{1}{12},$$
which is the desired assertion.

\section{Proof of Theorems \ref{T:2} and \ref{T:3}}
The entire section is devoted to the proof of Theorem \ref{T:2}. In the end we will give a short remark how to change the proof to get Theorem \ref{T:3}.

Fix an irrational $\alpha$ and numbers $c>0$, $\gamma\ge 2$ such that
\begin{equation}\label{E:6.1}
\bigg|\alpha - \frac{p}{q} \bigg| \le \frac{c}{q^\gamma}
\end{equation}
for infinitely many pairs $p,q \in \mathbb Z$, $q\not = 0$. Take $r$ to be the largest possible integer with $r<\frac{\gamma}{2}-\frac 3 2$. The function $s\longmapsto (\gamma-1)(1-s)-1$ is decreasing, $s\in [\frac 1 2, 1)$, and its value at $s=\frac 1 2$ is $\frac{\gamma}{2}-\frac 3 2$, thus by continuity we can choose $s>\frac 1 2$ such that $r<(\gamma-1)(1-s)-1$. For this choice of $s$ we are going to construct an observable $\varphi$ with $\int\varphi(x)dx=0$ such that
$$\mathbb P \bigg( \frac{\varphi(Y_1^{\alpha})+\cdots+\varphi(Y_n^{\alpha})}{n^s} >  \frac{\sqrt{2}}{4\cdot (16 c)^{1-s}} \bigg) > \frac{1}{16}$$
for infinitely many $n$'s. Consequently \CLT is violated.

Take arbitrary $p_1, q_1 \in \mathbb Z$, $q_1\not = 0$, satisfying (\ref{E:6.1}). Set $\varphi_1(x) = q_1^{-(\gamma-1)(1-s)}\cos(2 \pi q_1 x)$ and apply Lemma \ref{L:3} to get 
\begin{equation}\label{E:6.3}
\mathbb P \bigg( \frac{\varphi_1(Y_1^{\alpha})+\cdots+\varphi_1(Y_{N_1}^{\alpha})}{N_1^s} > \frac{\sqrt{2}}{2\cdot (16c)^{1-s}} \bigg) > \frac{1}{8},
\end{equation}
where $N_1=\lfloor\frac{q_1^{\gamma-1}}{16c}\rfloor$.
By Proposition \ref{P:2} the additive functional $(\varphi_1(Y^\alpha_1)+\cdots+\varphi_1(Y^\alpha_n))$ satisfies \textbf{CLT}, thus for $N$ sufficiently large
\begin{equation}\label{E:6.2}
\mathbb P \bigg( \frac{\varphi_1(Y_1^{\alpha})+\cdots+\varphi_1(Y_N^{\alpha})}{N^s} > \frac{\sqrt{2}}{4 \cdot (16c)^{1-s}}  \cdot \frac{1}{4} \bigg) < \frac{1}{8}\cdot\frac{1}{4}.
\end{equation}
Let us take $p_2, q_2\in \mathbb Z$, $q_2\not = 0$, such that (\ref{E:6.1}) holds, $N_2=\lfloor\frac{q_2^{\gamma-1}}{16c}\rfloor$ satisfies (\ref{E:6.2}) and
\begin{equation}
q_2^{-(\gamma-1)(1-s)} \cdot N_1^{1-s}<\frac{\sqrt{2}}{4 \cdot (16c)^{1-s}} \cdot \frac 1 4
\end{equation}
(this will imply that the inequality (\ref{E:6.3}) is not affected too much when $\varphi_1$ replaced by $\varphi_1+\varphi_2$). Lemma \ref{L:3} yields
$$\mathbb P \bigg( \frac{\varphi_2(Y_1^{\alpha})+\cdots+\varphi_2(Y_{N_2}^{\alpha})}{N_2^s} > \frac{\sqrt{2}}{2\cdot (16c)^{1-s}} \bigg) >\frac{1}{8}.$$

Assume $N_k$, $p_k$, $q_k$ are already defined. Let us choose a pair $q_{k+1}, p_{k+1} \in \mathbb Z$ with  (\ref{E:6.1}), where $q_{k+1}>q_k$ is so large that
\begin{equation}\label{E:6.5}
q_{k+1}^{-(\gamma-1)(1-s)} \cdot N_i^{1-s}< \frac{\sqrt{2}}{4\cdot (16c)^{1-s}} \cdot 4^{-(k-i)} \quad \textrm{for $i=1,\ldots, k$.}
\end{equation}
Moreover, using Lemma \ref{L:2} we demand that $q_{k+1}$ is so large that
\begin{equation}\label{E:6.4}
\mathbb P \bigg( \frac{\varphi_j(Y_1^{\alpha})+\cdots+\varphi_j(Y_{N_{k+1}}^{\alpha})}{N_{k+1}^s} > \frac{\sqrt{2}}{4 \cdot (16c)^{1-s}}\cdot \frac{1}{4^j} \bigg) < \frac{1}{8}\cdot\frac{1}{4^j} \quad \textrm{for $j \le k$,} 
\end{equation}
where $N_{k+1}=\lfloor\frac{q_{k+1}^{\gamma-1}}{16c}\rfloor$. Finally we use Lemma \ref{L:3} to get
\begin{equation}\label{E:6.6}
\mathbb P \bigg( \frac{\varphi(Y_1^{\alpha})+\cdots+\varphi(Y_N^{\alpha})}{N^s} > \frac{\sqrt{2}}{2\cdot (16 c)^{1-s}} \bigg) > \frac{1}{8},
\end{equation}
where $\varphi_{k+1}(x)=q_{k+1}^{-(\gamma-1)(1-s)}\cos (2\pi q_{k+1} x)$.
\noindent When the induction is complete put
$$\varphi(x)=\sum_{k=1}^\infty \varphi_k(x)=\sum_{k=1}^\infty q_k^{-(\gamma-1)(1-s)}\cos(2 \pi q_k x).$$
By assumption $r<(\gamma-1)(1-s)-1$, therefore we can take $\varepsilon>0$ so that $r=(\gamma-1)(1-s)-(1+\varepsilon)$. If one differentiates this series $r$ times, then it still converges uniformly (with the rate at least $q^{-(1+\varepsilon)}$). Therefore Theorem 7.17 (page 152) in \cite{Rudin_1976} yields $\varphi$ is $C^r$. 

Now it remains to show that
$$\mathbb P \bigg( \frac{\varphi(Y_1^{\alpha})+\cdots+\varphi(Y_{N_k}^{\alpha})}{N_k^s} >  \frac{\sqrt{2}}{4\cdot (16 c)^{1-s}} \bigg) > \frac{1}{16}$$
for every $k\in\mathbb N$. We proceed analogously to the proof of Theorem \ref{T:1}. Fix $k$. We have 
$$\frac{\varphi(Y_1^{\alpha})+\cdots+\varphi(Y_{N_k}^{\alpha})}{N_k^s} 
= \sum_{i\le k} \frac{\varphi_i(Y_1^{\alpha})+\cdots+\varphi_i(Y_{N_k}^{\alpha})}{N_k^s}$$
$$+\sum_{i> k} \frac{\varphi_i(Y_1^{\alpha})+\cdots+\varphi_i(Y_{N_k}^{\alpha})}{N_k^s}.$$
The application of (\ref{E:6.5}) yields the second term is bounded by $\frac{\sqrt{2}}{8\cdot (16 c)^{1-s}}$ a.s. Therefore
$$\mathbb P \bigg( \frac{\varphi(Y_1^{\alpha})+\cdots+\varphi(Y_{N_k}^{\alpha})}{N_k^s} >  \frac{\sqrt{2}}{4\cdot (16 c)^{1-s}} \bigg)$$
$$\ge \mathbb P \bigg( \sum_{i\le k} \frac{\varphi_i(Y_1^{\alpha})+\cdots+\varphi_i(Y_{N_k}^{\alpha})}{N_k^s} >  \frac{3\sqrt{2}}{8\cdot (16 c)^{1-s}} \bigg)$$
$$\ge \mathbb P \bigg( \frac{\varphi_k(Y_1^{\alpha})+\cdots+\varphi_k(Y_{N_k}^{\alpha})}{N_k^s} >  \frac{\sqrt{2}}{2\cdot (16 c)^{1-s}} \bigg)$$
$$-\sum_{i<k} \mathbb{P} \bigg( \frac{\varphi_i(Y_1^{\alpha})+\cdots+\varphi_i(Y_{N_k}^{\alpha})}{N_k^s} >   \frac{\sqrt{2}}{4\cdot (16 c)^{1-s}}\cdot \frac{1}{4^i} \bigg). $$
The application of (\ref{E:6.4}) and (\ref{E:6.6}) yields Theorem \ref{T:2}.

To demonstrate Theorem \ref{T:3} observe that for $\alpha$ Liouville there exist sequences of integers $p_k$, $q_k$ with
$$\bigg|\alpha - \frac{p_k}{q_k} \bigg| \le \frac{1}{q^k} \quad \textrm{for every $k$.}$$
The only difference with the proof of Theorem \ref{T:2} is that $p,q$ are chosen from this sequence. Then again $\varphi=\sum_m \varphi_m$, and the series is uniformly convergent after differentiating it $r$-times for an arbitrary $r$. This implies $\varphi\in C^\infty$. The rest remains unchanged.

\section{Acknowledgements}
This research was supported by the Polish National Science Centre grant Preludium UMO-2019/35/N/ST1/02363. I am grateful to Anna Zdunik for fruitful discussions and for sharing her notes with the proof of \CLT. I would also like to thank Corinna Ulcigrai for providing references \cite{Bromberg_Ulcigrai_2018, Sinai_Ulcigrai_2008}. I am grateful to two anonymous referees for many comments that helped improve the manuscript and for providing me reference \cite{Weber_2009}. Finally, I would like to thank Michael Lin for discussions and for pointing out that \cite{Kipnis_Varadhan_1986} applies here. 

\bibliographystyle{alpha}
\bibliography{CLT_rotations_JSP}

\end{document}